\documentclass[letterpaper, 10 pt, conference]{ieeeconf}  

\IEEEoverridecommandlockouts                              

\usepackage{graphicx}
\usepackage{graphics} 
\usepackage{amsmath} 
\usepackage{amssymb}  
\usepackage{amsfonts}
\usepackage{algorithm}
\usepackage{algpseudocode}
\usepackage{tabu}
\usepackage{color}
\let\labelindent\relax
\usepackage{enumitem}

\newtheorem{theorem}{\bf Theorem}

\newtheorem{definition}{\bf Definition}

\newtheorem{lemma}{\bf Lemma}

\DeclareMathOperator{\argmin}{argmin}

\newcommand*{\LONGVERSION}{}

\title{\LARGE \bf
Incentivizing Truth-Telling in MPC-based Load Frequency Control
}

\author{Takashi Tanaka$^{1}$ and Vijay Gupta$^{2}$
\thanks{$^{1}$ACCESS Linnaeus Centre, KTH Royal Institute of Technology, 100 44 Stockholm, Sweden
        {\tt\small ttanaka@kth.se}}%
\thanks{$^{2}$Department of Electrical Engineering, University of Notre Dame, Notre Dame, IN 46556, United States
        {\tt\small vgupta2@nd.edu}}%
}

\begin{document}

\maketitle
\thispagestyle{empty}
\pagestyle{empty}

\begin{abstract}
We present a mechanism for socially efficient implementation of model predictive control (MPC) algorithms for load frequency control (LFC) in the presence of self-interested power generators.
Specifically, we consider a situation in which the system operator seeks to implement an MPC-based LFC for aggregated social cost minimization, but necessary information such as individual generators' cost functions is privately owned. 
Without appropriate monetary compensation mechanisms that incentivize truth-telling,  self-interested market participants may be inclined to misreport their private parameters in an effort to maximize their own profits, which may result in a loss of social welfare. The main challenge in our framework arises from the fact that every participant's strategy at any time affects the future state of other participants; the consequences of such dynamic coupling has not been fully addressed in the literature on online mechanism design.  We propose a class of real-time monetary compensation schemes that incentivize market participants to report their private parameters truthfully at every time step, which enables the system operator to implement MPC-based LFC in a socially optimal manner. 
\end{abstract}

\section{Introduction}

Load frequency control (LFC) regulates the power flow between different areas in the power grid to minimize transient frequency deviation and ensuring steady state frequency deviation to be zero. The power flow should ideally be done in a way that minimizes the operating cost of the power system. By its nature, LFC is a large-scale and highly complex optimal control problem that is typically solved using a hierarchical architecture \cite{ilic2007hierarchical,dorfler2014breaking}. In a deregulated power system, those generators (more precisely, their owners) that provide LFC services have to be compensated appropriately by the system operator. It is, therefore, important to design appropriate market mechanisms in which the collective profit-maximizing behaviors of such entities are aligned with the solution that maximizes the social welfare. Efficient market design for LFC particularly as renewable penentration increases continues to see active work. For one, the additional stochasticity that renewables bring to the grid means extra disturbances in frequency that the LFC needs to suppress \cite{bitar2012bringing}. Further, physical inertia of the conventional generators, which is an integral part of the frequency control loop, is reduced as more and more  conventional generators are replaced with converter-based power suppliers \cite{doherty2010assessment,morren2006inertial}.  To maintain system frequency even in the face of these challenges, new control and real-time optimization algorithms, \emph{together with appropriate market mechanisms to incentivize companies to provide LFC services}, will be needed.


A control algorithm that has recently been proposed for LFC is Model Predictive Control (MPC), also known as Receding Horizon Control (RHC)~\cite{ulbig2013predictive,venkat2008distributed}. MPC is a sub-optimal control algorithm, which determines control inputs by  solving a finite horizon open-loop optimal control problems (i.e., optimization problems) repeatedly at every time step~\cite{camacho2013model}. Advantages of MPC include the flexibility to take various system constraints into account explicitly and  
the fact that computational tractability is obtained at the expense of only moderate performance loss of control. Due to these advantages, MPC has been proposed for several different purposes in power system operations~\cite{xie2012fast,arnold2011model,venkat2008distributed,
khalid2009model,ulbig2013predictive}. In particular, it has been proposed for real-time frequency control of the synchronous grid in order to compensate inertia loss due to the increasing penetration of distributed generation units \cite{ulbig2013predictive}.  Distributed versions of MPC algorithm have also been proposed for automatic generation control in \cite{venkat2008distributed}. 


However, most of the existing literature assumes that the MPC algorithm is implemented by the central system operator. Even works that consider a distributed setting, the individual decision makers are assumed to be fully cooperative with the central system operator. In a realistic situation in a deregulated power market, centralized solutions are not feasible. Further, the individual  decision makers that own the generators are strategic entities that seek to optimize their profits, rather than to collaborate with the central  system operator. For ensuring the wide-spread applicability of MPC for frequency control in deregulated power systems, many issues, thus, need to be addressed:
\begin{itemize}[leftmargin=3ex]
\item[(A)] Information-asymmetry: Important system parameters, such as generation cost functions, are private information and owned by the distributed decision-makers in the system.
\item[(B)] Faithful implementation of control actions: Control actions are executed by these strategic decision makers who are market participants, and not directly by the system operator. Unless the central system operator has direct control of power system resources, an incentive is needed for market participants to faithfully execute appropriate control actions  \cite{berger1989real,tanaka2012dynamic}.
\item[(C)] Distributed computation of optimal control actions: Even for calculating the optimal control actions, distributed computation and communication among  participants may be required (e.g., \cite{venkat2008distributed}). Usually, incentives are needed for market participants to faithfully implement the prescribed distributed computation and communication algorithms \cite{tanaka2013faithful}. 
\end{itemize}
In other words, to utilize MPC for load frequency control, there is a need to design a market mechanism that incentivizes the companies that own generators and are strategic market participants to (A) report their private parameters to the system operator truthfully, (B) execute socially optimal control actions faithfully, and (C) execute distributed algorithms faithfully. Incorrect consideration of the strategic behaviors from the participants can lead to control actions that are not globally optimal, leading to loss of social welfare. Since the power market is usually an oligopoly, this phenomenon is similar to the well known fact that Cournot equilibrium is, in general, different from the competitive (and socially optimal) equilibrium \cite{mas1995microeconomic} in oligopoly. In fact, several examples of market failures and ensuing power system disturbances (California in 2000 and 2001, Texas in 2005, and New York City between 2006 and 2008 \cite{vaheesan46market}) are now revealed to be, at least in part, due to underestimated market power of the participants in the power markets.

In this paper, we design a market mechanism that mitigates market participants' incentives for strategic behaviors when the system operator wishes to implement a model predictive control algorithm. For simplicity, in this paper, we focus on item (A) -- information assymetry -- alone; items (B) and (C) are addressed elsewhere. We assume that the system operator acts as a central planner  who is responsible for solving MPC in a real-time manner, and has the authority to command all generators' control inputs. However, to calculate the socially optimal solution, the operator needs to know the private parameters (such as costs) from the various participants. Thus, the challenge is to incentivize market participants to report their private parameters truthfully for every iteration of the MPC algorithm (i.e., at every time step) so that the central planner is able to solve for the socially optimal solution.

Truthful reporting of private information has been considered in mechanism design theory, which is an extensively studied branch of microeconomic theory \cite[Ch.23]{mas1995microeconomic}. Indeed, applying mechanism design theory to power system operation is not a new attempt \cite{hobbs2000evaluation, samadi2012advanced, tang2011stochastic}. However, most such attempts are currently restricted to static (single-stage) mechanism design problems. Since real-time market-based power system operation involves repetitive auctions, dynamic extensions of these results are desired. 

The framework of online mechanism design can be found in, e.g.,   \cite{parkes2004mdp,bergemann2010dynamic,pavan2014dynamic} and references therein. However, the problem has been solved under specific assumptions such as the state evolution of the agents being independent of each other. 
In the specific setting of our paper, the main challenge arises from the fact that every participant's strategy at any time affects the future state of other participants; such dynamic coupling has not been fully addressed in online mechanism design literature; see \cite{cavallo2012optimal} for a related discussion.  Related works include~\cite{tanaka2012dynamic} where an online mechanism design method was proposed to incentivize strategic power generators to execute optimal control actions at every time instance (item (B) in the list above). It was subsequently shown in \cite{tanaka2013faithful2} that a similar online mechanism can also be used for distributed computation and communication actions (item (C) in the list above). This result is discussed in a general framework of \emph{indirect} mechanism design theory in \cite{tanaka2013faithful}, where sufficiency and necessity of VCG-like mechanism is presented. Allied works also include~\cite{okajima2013dynamic,murao2014dynamic} which consider stochastic (LQG) power systems and propose Bayesian incentive compatible mechanisms. 

In contrast to \cite{tanaka2012dynamic,tanaka2013faithful2,tanaka2013faithful}, the goal of this paper is to present how a mechanism design framework can be used to address item (A)  with the motivating application of MPC-based LFC. 
We first introduce the notion of $\epsilon$-incentive compatibility, under which no strategic generator can gain more than $\epsilon$ by misreporting private information.
Then, we propose a VCG-like online mechanism that implements MPC with  $\epsilon$-incentive compatibility.
Finally, we analyze how the horizon length $T$ of the considered MPC affects $\epsilon$  in the proposed mechanism. An explicit relationship between $T$ and $\epsilon$ is obtained, in particular, 
 for LQ control problems. We note that receding horizon mechanisms in the context of power systems were considered in \cite{okajima2013dynamic,murao2014dynamic}, but the relationship between horizon length and $\epsilon$-incentive compatibility was not presented there.

The paper is organized as follows. We begin in Section~\ref{secexample} with a motivating example showing that in MPC implementations of LFC, strategic behavior by the market participants can easily lead to loss of social welfare. The problem is formulated in Section~\ref{sec:formulation}. The online mechanism is presented as a solution in Section~\ref{sec:vcg}. The results are illustrated with the special case of LQ optimal control in Section~\ref{secperformance}. Section~\ref{sec:conclusions} concludes the paper and presents some avenues for future work.

%



\section{Motivating Example}
\label{secexample}

\begin{center}
\begin{table}
  \caption{Tie line stiffness $T_{12}=1$, time step $\Delta_{\text{sample}}=0.1$.}
  \label{table:parameters}
{\tabulinesep=0.6mm
\begin{center}
  \begin{tabu}{| c | c  | c |}
   \hline
Angular momentum & $M_1=3.5$ & $M_2=4$  \\ 
$\frac{\text{Percent change in load}}{\text{Percent change in frequency}}$ & $D_1=2$ & $D_2=2.75$ \\
Charging time constant & $T_{\text{CH}_1}=50$ & $T_{\text{CH}_2}=10$ \\
$\frac{\text{Percent change in frequency}}{\text{Percent change in unit output}}$ & $R^{\text{f}}_1=0.03$ & $R^{\text{f}}_2=0.07$ \\
Governor time constant & $T_{\text{G}_1}=40$ & $T_{\text{G}_2}=25$ \\
\hline
  \end{tabu}
\end{center}
}
  \end{table}
\end{center}

\begin{figure}[t]
\centering
    \includegraphics[width=\columnwidth]{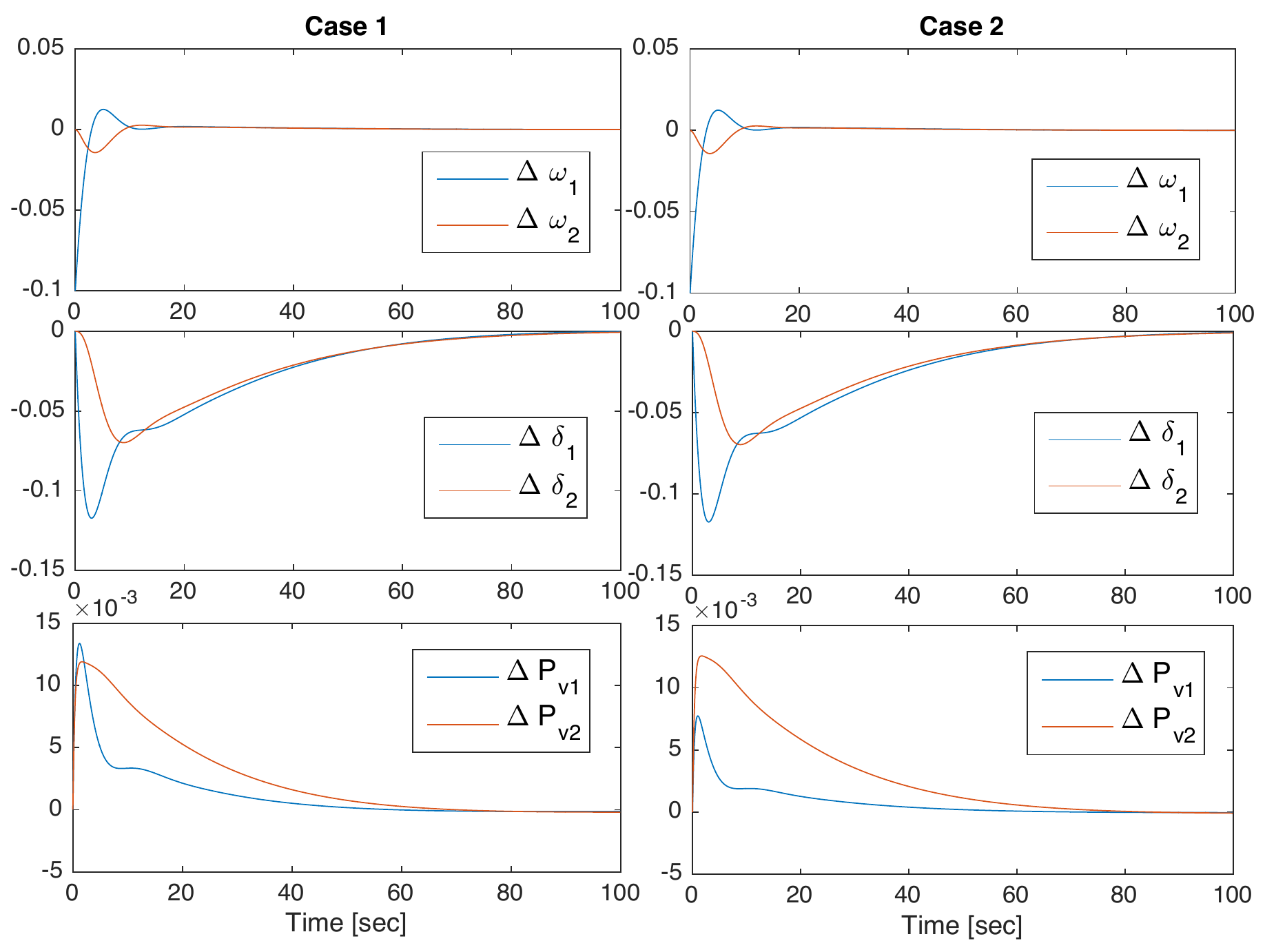} 
    \caption{Case 1: LFC in which both generators report private cost function truthfully. Case 2: LFC in which generator 1 misreports.}
    \label{fig:comparison}
\end{figure}

\begin{center}
\begin{table}
  \caption{Performance comparison.}
  \label{table:performance}
{\tabulinesep=0.6mm
\begin{center}
  \begin{tabu}{| c | c  | c | c |}
   \hline
   & Cost for area 1& Cost for area 2 &  Total cost \\
   \hline
Case 1 & $24.64$ & $18.01$ & $42.65$  \\ 
Case 2 & $23.83$ & $19.62$ & $43.45$  \\ 
\hline
  \end{tabu}
\end{center}
}
  \end{table}
\end{center}

Without appropriate monetary compensation mechanisms, a strategic market participant (e.g., a generator) may misreport its private parameters, if such false reports induces the system operator to command a control input that is globally sub-optimal but incurs smaller cost for the misreporting participant. In an LFC framework, we consider the following illustrative example. Consider a situation in which two power generation firms, say firm 1 and firm 2, own all generators in areas 1 and 2, respectively. For simplicity, we assume that all generator turbines in each area are completely synchronized, and hence modeled as a single large turbine. Assume that there is a tie line between two areas over which two areas can exchange power.  A simplified swing equaition linearized around its nominal operation point, borrowed from \cite{venkat2008distributed},  is given by (for $i,j=1,2$)\footnote{For implementation of MPC, the state space model is converted a discrete-time model.}
\begin{subequations}
\begin{align}
&\frac{d\Delta \omega_i}{dt}=-\frac{D_i}{M_i}\Delta \omega_i\!+\!\frac{1}{M_i}\Delta P_{\text{mech}_i}\!-\!\frac{T_{\text{tie}}}{M_i}(\Delta \delta_i\!-\!\Delta \delta_j) \\
&\frac{d\Delta P_{\text{mech}_i}}{dt}=-\frac{1}{T_{\text{CH}_i}}\Delta P_{\text{mech}_i}+\frac{1}{T_{\text{CH}_i}}\Delta P_{\text{v}_i} \\
&\frac{d\Delta P_{\text{v}_i}}{dt}=-\frac{1}{T_{\text{G}_i}}\Delta P_{\text{v}_i}-\frac{1}{R_i^{\text{f}}T_{\text{G}_i}}\Delta \omega_i + \frac{1}{T_{\text{G}_i}}\Delta P_{\text{ref}_i} \\
&\frac{d\Delta \delta_i}{dt}=\Delta \omega_i.
\end{align}
\end{subequations}
This is a dynamical system with eight dimensional state vector $x_t=(x^1_t, x^2_t)$ where
$
x_t^i=(\Delta \omega_i, \Delta P_{\text{mech}_i}, \Delta P_{\text{v}_i},\Delta \delta_i).
$
For each area $i=1,2$, 
 $\Delta \omega_i$ is the frequency deviation from the nominal value, $\Delta P_{\text{mech}_i}$ is the deviation in mechanical power, $\Delta P_{\text{v}_i}$ is the deviation in steam valve position, $\Delta \delta_i$ is the deviation in mechanical angle.
Command input signal 
$
u_t^i=\Delta P_{\text{ref}_i}
$
for each $i=1,2$ is the control input. Other system parameters used in the simulation are also borrowed from \cite{venkat2008distributed} and are summarized in Table \ref{table:parameters}.
Suppose that firm $i$'s generation cost  is modeled by 
\[
J^i=\sum_{t=1}^\infty \left(\|x_t\|_{Q^i}^2+\|u_t\|_{R^i}^2\right)
\] 
where $Q^i$ and $R^i$ are weight matrices only known to firm $i$. In this example, we assume that true values of these weight matrices are given by
\[
Q^i=\text{diag}(10, 1, 500, 10), \;\; R^i=0.1
\]
for both $i=1,2$. Since these are private variables, each firm needs to send this data to the system operator. The system operator schedules a socially optimal control input minimizing $J^1+J^2$ based on the reported information. Notice that once private variables $Q^i$ and $R^i$ are available, this is the standard Linear-Quadratic-Regulator (LQR) problem.

As a case study, we consider a situation in which a frequency deviation $\omega_1(0)=-0.1$ occurs in area 1 at time $t=0$, and study how the LFC recovers the nominal operation points.
\subsubsection{Case 1} This is the reference scenario in which both firms report their private variables truthfully and hence the system operator is able to schedule the socially optimal control sequence. Figure~\ref{fig:comparison} (left) shows the dynamic response of the frequency deviation in each area, deviation in mechanical angle in each area, and deviation in steam valve position in each area. The first line of Table~\ref{table:performance} summarizes the simulated value of $J^1$ and $J^2$ as well as the total cost $J^1+J^2$.

\subsubsection{Case 2} In this scenario, we suppose that firm 1 misreports its private value and the central planner designs a control law based on the reported false information. In particular, suppose that firm 1 reports $Q^1=\text{diag}(10, 1, 1000, 10)$, pretending that moving steam valve position in area 1 is more costly than the reality. Figure~\ref{fig:comparison} (right) shows the dynamic response in this scenario. Notice that steam valve position in area 1 is kept low compared to Case 1. 
Of course, the control signal in this scenario does not attain social optimality since it is generated using false information. It can be seen in Table~\ref{table:performance} that the value of $J^1 +J^2$ has increased. However, notice that $J^1$ alone has decreased from Case 1 to Case 2, indicating that \emph{firm 1 has an incentive to misreport its private variables}. 
This simple example suggests that, without an appropriate monetary compensation mechanism, one cannot expect truthful reports from the strategic market participants.

\section{Problem Formulation}
\label{sec:formulation}
To formally describe our problem, suppose that there exist a central system operator and $N$ strategic market participants (i.e., generators) that seek to provide frequency control services and be compensated.
The plant to be controlled is given by
\begin{equation}
\label{eqplant}
x_{t+1}=f(x_t, u_t), \;\; t=0,1,\cdots .
\end{equation}
Suppose that the state $x_t$ of the system is fully observable by the system operator.
Without loss of generality, we assume that $x_t=0$ is the nominal operating point of the system.
The market participants are indexed by $i\in\{1,\cdots, N\}$, and the state and control vectors are concatenations of their ``local states" and ``local controls"
\[x_t=(x_t^1,\cdots, x_t^N), \;\; u_t=(u_t^1,\cdots, u_t^N).\]
Suppose that the control cost $c_t^i(x_t, u_t)$ is incurred to the $i$-th participant at time step $t$.
The purpose of the system operator is to minimize the aggregated cost by solving an optimal control problem
\begin{equation}
\label{eqtrueoptcontprob}
\min_{u_{0:\infty}} \; \; \sum_{t=1}^\infty\sum_{i=1}^N c_t^i(x_t^i, u_t^i)
\end{equation}
subject to the dynamical equation (\ref{eqplant}), as well as local state and control constraints
\[
 x_t^i \in\mathcal{X}_t^i,   u_t^i \in\mathcal{U}_t^i, \;\;\; i=1,\cdots, N, \;\; t=1,2,\cdots
\]
In (\ref{eqtrueoptcontprob}), $u_{0:\infty}$ is a short-hand notation for $(u_0,u_1,\cdots)$. 
Similar notations will be used in the sequel.

\subsection{MPC algorithm}
Solving (\ref{eqtrueoptcontprob}) directly may be computationally expensive as well as not useful since parameters such as fuel prices may change over time. Hence, we assume the system operator repeatedly solves the following open-loop optimal control problem at every $t$:
\begin{subequations}
\label{eqmpc}
\begin{align}
 \min\;\; & \sum_{k=t}^{t+T} \sum_{i=1}^N c_k^i(x_k^i, u_k^i) \\
\text{s.t. }\;\;&  x_{k+1}=f_k(x_k, u_k) \label{eqmpc1}\\
& x_k^i \in\mathcal{X}_k^i,   u_k^i \in\mathcal{U}_k^i \\
& \forall k\in\{t,\cdots,t+T\}, \forall i \in\{1,\cdots, N\},
\end{align}
\end{subequations}
where $T$ is the horizon length.
Note that at every time step the system operator observes $x_t$, which is used as the initial condition for (\ref{eqmpc1}).

\subsection{Problem setup}
Suppose that the function $f_t$, state constraint set $\mathcal{X}_t^i$ and the control constraint set $\mathcal{U}_t^i$ are publicly known but the cost function $c_t^i$ is
private and only known to the $i$-th agent.
Private information must be truthfully reported to the central planner in advance so that the open-loop optimal control problem (\ref{eqmpc}) is correctly solved. 
In particular, we denote by $\theta_t^i$ the \emph{type} of the $i$-th agent, which fully describes the function $c_t^i$.
In the example in Section~\ref{secexample}, $\theta_t^i$ was a weight matrix $Q^i(=Q_t^i=Q_{t+1}^i=\cdots)$.
The type vector $\theta_t^i$ must be truthfully reported to the central planner at time step $t$.
Suppose that the $i$-th agent is interested in minimizing his/her own cost $\sum_{t=0}^\infty c_t^i(x_t^i, u_t^i)$ rather than aggregated social cost (\ref{eqtrueoptcontprob}). 
As we have seen in Section~\ref{secexample}, strategic agents may be incentivized to misreport their types in an effort to minimize their own cost. 

\subsection{Disturbance model}
For simplicity, we assume that LFC system is subject to an impulse frequency disturbance at $t=0$, and this is modeled by a non-zero initial condition $x_0 \neq 0$. We assume that system operator and market participants have no knowledge about probability distribution of $x_0$. Our goal is to design a mechanism that induces truth-telling regardless of the realization of $x_0$. (In other words, in this paper we employ the solution concept of \emph{ex post} incentive compatibility.)
Note that impulse or step signals are commonly used disturbance models in frequency control (e.g., \cite{dorfler2014breaking, andreasson2014distributed}). 

\section{Proposed Online Mechanism}
\label{sec:vcg}
To achieve the aforementioned goal, we propose a real-time mechanism inspired by the online Vickrey-Clarke-Groves (VCG) mechanism \cite{bergemann2010dynamic}.
Let $u_{t:\infty}$ be a sequence of control inputs, and $x_{t:\infty}$ be the induced state trajectory.
Denote by
\begin{subequations}
\begin{align}
J_t^i (u_{t:\infty}^i;\theta_{t:\infty}^i)&\triangleq \sum_{k=t}^\infty c_k^i(x_k^i,u_k^i) \label{defji}\\
J_t(u_{t:\infty};\theta_{t:\infty})&\triangleq \sum_{i=1}^N J_t^i (u_{t:\infty}^i;\theta_{t:\infty}^i). \label{defj}
\end{align}
\end{subequations}
the cost-to-go functions. 
Note that the dependence of the cost-to-go function on $\theta_{t:\infty}$ reflects the fact that it is evaluated using private information $c_k^i$ contained in $\theta_{t:\infty}$.
The  subroutine $\textsf{OPENLOOP}_T(\cdot)$ shown in Algorithm \ref{alg:MPC} summarizes how our MPC algorithm is executed at every time step.
At every $t$, this subroutine receives reported type vectors $\theta_t$ from the agents. Based on the received information, an open-loop optimal control problem (\ref{eqmpc}) is formulated. 
By solving (\ref{eqmpc}) numerically, the subroutine returns a control action $u_t=\textsf{OPENLOOP}_T(\theta_t)$ to be implemented at the current time step.

\begin{algorithm}[htb]
        \caption{$\textsf{OPENLOOP}_T(\cdot)$}
        \textbf{Input:} Type vectors $\theta_t=(\theta_t^1,\cdots,\theta_t^N)$ reported by the agents, which contains necessary information to formulate an open-loop optimal control problem (\ref{eqmpc}).

        \textbf{Output:} Control input $u_t$ at current time step $t$.
        \begin{algorithmic}[1]
\State{Formulate an open-loop optimal control problem (\ref{eqmpc}).}
\State{Solve (\ref{eqmpc}) to obtain an optimal control sequence $u_{t:t+T}$.}
\State{Discard $u_{t+1:t+T}$ and return $u_t$.}
        \end{algorithmic}
        \label{alg:MPC}
\end{algorithm}

The MPC scheme with the horizon length $T$ is a sequential execution of this subroutine
\begin{align*}
u_t&=\textsf{OPENLOOP}_T(\theta_t) \\
u_{t+1}&=\textsf{OPENLOOP}_T(\theta_{t+1}) \\
&\;\; \vdots
\end{align*}
This sequence of equations is denoted by $u_{t:\infty}=\textsf{MPC}_T(\theta_{t:\infty})$. 
Note that $\textsf{MPC}_T(\cdot)$ maps reported type vectors $\theta_{t:\infty}$ to a social decision $u_{t:\infty}$.
In the terminology of mechanism design, such a map is called a \emph{decision rule}.
\begin{definition}
A decision rule $\textsf{MPC}_T(\cdot)$ is said to be \emph{$\epsilon$-efficient} if, for all $t$ and $x_t$,
\[
J_t(\textsf{MPC}_T(\theta_{t:\infty});\theta_{t:\infty})\leq J_t(u_{t:\infty};\theta_{t:\infty})+\epsilon
\]
for all $u_{t:\infty}\in \prod_{k=t}^\infty\prod_{i=1}^N \mathcal{U}_{k,i}$ and $\theta_{t:\infty}\in\prod_{k=t}^\infty\prod_{i=1}^N \Theta_{k,i}$.
\end{definition}

Notice that, in general, $u_{t:\infty}=\textsf{MPC}_T(\theta_{t:\infty})$ is only a suboptimal control sequence. The notion of $\epsilon$-efficiency guarantees that the performance loss of the MPC from the globally optimal control strategy is bounded by $\epsilon$. 
In Section \ref{secperformance}, we will present an example of an LQ control problem in which $\epsilon$ can be obtained as a function of $T$.

We propose an online (multistage) mechanism $\mathcal{M}$ that can be used by the system operator to induce truthful reports by market participants at every time step. 
In this mechansim, the system operator introduces a real-time tax scheme (the amount of taxes charged to individual market participants are calculated according to some rule based on participants' reports), which creates an $N$-player multistage game. 
The real-time tax rule is carefully designed so that truth-telling by all participants corresponds to a subgame perfect equilibrium \cite{shoham2008multiagent} of the resulting multistage game. 

Formally, we consider $\mathcal{M}$ as a collection of ``submechanisms"  $\mathcal{M}(t,x_t)$ parameterized by time $t$ and the state $x_t$.
At every $(t,x_t)$, a submechanism $\mathcal{M}(t,x_t)$ accepts participants' reports $(\theta_t^1,\cdots,\theta_t^N)$, computes control input $u_t$, determines tax values charged to individual participants at time $t$, and sends the system to the next state $(t+1, x_{t+1})$. 
This way, a submechanism defines a subgame, and a collection of submechanisms defines a multistage game.

We assume that each player's cost is quasilinear, i.e., a summation of control-related cost (i.e., cost function $c_t^i$) and the amount of tax (denoted by $p_t^i$) charged by the system operator.
Under this assumption, submechanisms $\mathcal{M}(t,x_t)$ are fully specified once we fix how control inputs $u_t$ and tax values $p_t^i$ are determined based on the reported information.
Among many possible designs, we propose a particular design $\mathcal{M}_{\text{VCG-MPC}}$ specified by the following scheme.
\begin{enumerate}
\item A control action to be executed at the current time step is $u_t(\theta_t)=\textsf{OPENLOOP}_T(\theta_t)$.
\item Tax values to be charged to the $i$-th agent is $p_t^i(\theta_t)=\sum_{j\neq i} c_t^j(x_t^j, u_t^j(\theta_t))+K_t^i$, where $K_t^i$ is a quantity that does not depend on the history of the $i$-th agent's reports $(\theta_1^i,\cdots, \theta_t^i)$ calculated by a publicly know rule.
\end{enumerate}
Notice that this choice is motivated by the structure of the VCG mechanism (e.g., \cite{mas1995microeconomic}).

\subsection{Incentive compatibility}
Denote by $\pi_t^i(\theta_{t:\infty})\triangleq \sum_{k=t}^\infty p_k^i(\theta_k)$ the ``tax-to-go" function for the $i$-th agent at time $t$.
\begin{definition}
For a given time-state pair $(t,x_t)$, $\mathcal{M}(t,x_t)$ is said to be $\epsilon$-incentive compatible if for every $i$, $\theta_{t:\infty}\in\prod_{k=t}^\infty\prod_{i=1}^N \Theta_k^i$, and $\hat{\theta}_{t:\infty}^i \in \prod_{k=t}^\infty\Theta_k^i$, we have
\begin{align*}
&J_t^i(\textsf{MPC}_T(\theta_{t:\infty});\theta_{t:\infty}^i)+\pi_t^i(\theta_{t:\infty}) \\
&\leq J_t^i(\textsf{MPC}_T(\hat{\theta}_{t:\infty}^i,\theta_{t:\infty}^{-i});\theta_{t:\infty}^i)+\pi_t^i(\hat{\theta}_{t:\infty}^i, \theta_{t:\infty}^{-i})+\epsilon.
\end{align*}
A mechanism $\mathcal{M}$ is said to be $\epsilon$-incentive compatible if its submechanisms are all $\epsilon$-incentive compatible\footnote{Superscript ``$-i$" indicates the collection of agents excluding $i$. With an abuse of notation, we also write a vector $\theta$ as $(\theta^i, \theta^{-i})$.}.

\end{definition}

The $\epsilon$-incentive compatibility is a significant property of a mechanism guaranteeing that, at any time $t$ and any state $x_t$ of the system, no participant can find a false report sequence $\hat{\theta}_t^i, \hat{\theta}_{t+1}^i, \cdots$ for the future that reduce her net cost more than $\epsilon$ compared to the case in which she makes a truthful sequence of reports $\theta_t^i, \theta_{t+1}^i, \cdots$, and this holds true regardless of the other players' true cost functions and their reporting strategies. Hence, if $\epsilon$ can be made sufficiently small, no agent has a strict incentive to misreport her private parameters.

\begin{theorem}
\label{maintheo}
If $\textsf{MPC}_T(\cdot)$ is $\epsilon$-efficient, then for every time-state pair $(t,x_t)$, $\mathcal{M}_{\text{VCG-MPC}}(t,x_t)$ is $\epsilon$-incentive compatible. 
\end{theorem}
\begin{proof}
Suppose there exists $t,i,\theta_{t:\infty}$ and $\hat{\theta}_{t:\infty}^i$ such that 
\begin{align*}
&J_t^i(\textsf{MPC}_T(\theta_{t:\infty});\theta_{t:\infty}^i)+\pi_t^i(\theta_{t:\infty}) \\
&> J_t^i(\textsf{MPC}_T(\hat{\theta}_{t:\infty}^i,\theta_{t:\infty}^{-i});\theta_{t:\infty}^i)+\pi_t^i(\hat{\theta}_{t:\infty}^i, \theta_{t:\infty}^{-i})+\epsilon.
\end{align*}
Substituting (\ref{defji}) and the expression of $p_k^i(\theta_k)$ in 2),
\begin{align*}
&\sum_{k=t}^\infty c_k^i(x_k^i,u_k^i)+\sum_{k=t}^\infty\sum_{j\neq i}c_k^j(x_k^j,u_k^j) +\sum_{k=t}^\infty K_k^i\\
&> \sum_{k=t}^\infty c_k^i(\hat{x}_k^i,\hat{u}_k^i)+\sum_{k=t}^\infty\sum_{j\neq i}c_k^j(\hat{x}_k^j,\hat{u}_k^j) +\sum_{k=t}^\infty K_k^i+\epsilon
\end{align*}
where $u_{t:\infty}=\textsf{MPC}_T(\theta_{t:\infty})$, $\hat{u}_{t:\infty}=\textsf{MPC}_T(\hat{\theta}_{t:\infty}^i,\theta_{t:\infty}^{-i})$, and $x_{t:\infty}$ and $\hat{x}_{t:\infty}$ are trajectories driven by $u_{t:\infty}$  and $\hat{u}_{t:\infty}$ respectively, starting from $x_t=\hat{x}_t$.
Using (\ref{defj}), this can be rearranged as
\begin{align*}
&J_t(\textsf{MPC}_T(\theta_{t:\infty});\theta_{t:\infty})+\sum_{k=t}^\infty K_k^i \\
&>J_t(\textsf{MPC}_T(\hat{\theta}_{t:\infty}^i,\theta_{t:\infty}^{-i});\theta_{t:\infty})+\sum_{k=t}^\infty K_k^i+\epsilon.
\end{align*}
Since $\sum_{k=t}^\infty K_{k,i}$ does not depend on $\theta_i$,
\[
J_t(\textsf{MPC}_T(\theta_{t:\infty});\theta_{t:\infty})>J_t(\textsf{MPC}_T(\hat{\theta}_{t:\infty}^i,\theta_{t:\infty}^{-i});\theta_{t:\infty})+\epsilon.
\]
However, this contradicts to the $\epsilon$-efficiency of $\textsf{MPC}_T(\cdot)$.
\end{proof}

\subsection{Choice of free parameters}
Proposed mechanism $\mathcal{M}_{\text{VCG-MPC}}$ has a large degree of freedom in the choice of function $K_t^i$. In the standard framework of Groves mechanisms, this degree of freedom is used to achieve other desirable properties, such as \emph{budget balance}  and \emph{individual rationality} \cite[Ch.23]{mas1995microeconomic}.
In the context of online mechanisms, \cite{bergemann2010dynamic} proposes to choose $K_t^i$ as
\begin{equation}
\label{eqvcgk}
K_t^i=-\sum_{j\neq i} \text{``$j$'s cost at $t$ when $i$ is absent"}.
\end{equation}
In the problem formulated in \cite{bergemann2010dynamic}, the tax calculated by $p_t^i(\theta_t)=\sum_{j\neq i} c_t^j(x_t^j, u_t^j(\theta_t))+K_t^i$ with (\ref{eqvcgk}) matches the \emph{flow marginal contribution} (marginal contribution at specific time instances) of agent $i$ to the rest of the society, and the resulting online mechanism possesses some desirable properties.

Unfortunately, due to several important differences between our problem setting and that of \cite{bergemann2010dynamic}\footnote{In our problem formulation, participant $i$'s strategy at time $t$ affects the future state of other participants $x^j_k$, $k> t, j\neq i$. Such a dynamic coupling does not appear in the model of \cite{bergemann2010dynamic}.}, evaluating (\ref{eqvcgk}) in our scenario is a much more complicated task. 
To evaluate the marginal contribution by the $i$-th agent, we need to consider a situation in which the $i$-th agent is absent from the market.
One possible approach is to force the $i$-th control input to be zero over the entire time horizon. 
In this case, the society's cost with the $i$-th agent being absent is captured by the following optimal control problem:
\begin{subequations}
\label{eqmpcmarginal}
\begin{align}
 \min\;\; & \sum_{t=1}^\infty \sum_{j\neq i} c_k^j(x_k^j, u_k^j) \\
\text{s.t. }\;\;&  x_{t+1}=f_t(x_t, u_t) \label{eqmpcmarginal1}\\
& x_t^j \in\mathcal{X}_t^j,   u_t^j \in\mathcal{U}_t^j \\
& \forall t\in\{1, 2, \cdots \}, \forall j \in\{1,\cdots, N\} \\
& u_t^i =0 \;\; \forall t \in \{1,2,\cdots \}.
\end{align}
\end{subequations}
Denote by an optimal solution to (\ref{eqmpcmarginal}) by $x[-i]$, $u[-i]$.
Then it is reasonable to evaluate (\ref{eqvcgk}) by
\begin{equation}
\label{eqvcgk1}
K_t^i=-\sum_{j\neq i} c_t^j (x[-i]_t^j, u[-i]_t^j).
\end{equation}
Notice that $K_t^i$ defined this way does not depend on the $i$-th agent's reporting strategy, since solving (\ref{eqmpcmarginal}) does not require the knowledge of $\theta_1^i, \theta_2^i, \cdots$. Thus the resulting mechanism attains $\epsilon$-incentive compatibility, as per Theorem \ref{maintheo}.
However, whether the choice (\ref{eqvcgk1}) has an advantage (in terms of, e.g., individual rationality or budget balance) in the considered LFC problem is currently unknown.

\section{A Special Case}
\label{secperformance}
In this section, we consider a special case in which plant (\ref{eqplant}) is linear time-invariant (LTI), the cost functions are quadratic, and there are no state and control constraints.
The purpose of this section is to show a concrete example in which MPC is $\epsilon$-efficient and hence the resulting mechanism is $\epsilon$-incentive compatible.
Using techniques in \cite{nevistic1997finite}, we also analyze how the horizon length $T$ affects the $\epsilon$-incentive compatibility.
Such an explicit analysis may not be possible for more practical MPCs, but simple observations in this section provide valuable intuition for more complex cases.  

Consider an infinite horizon optimal control problem
\[
\min_{u_{0:\infty}} \; \sum_{t=0}^\infty  \left(x_t^\top Q_t x_t+u_t^\top R_t u_t\right)
\]
with $Q_t=\text{diag}(Q_t^1, \cdots, Q_t^N)$, $R_t=\text{diag}(R_t^1, \cdots, R_t^N)$, subject to a linear plant equation $x_{t+1}=Ax_t+Bu_t$ with some given initial state $x_0$. Assume $(A,B)$ is a stabilizable pair.
Denote the cost-to-go function by
\begin{equation}
\label{eqoptcost}
J_t(x_t)=\sum_{k=t}^\infty \left(x_k^\top Q_k x_k+u_k^\top R_k u_k\right).
\end{equation}
Suppose that $\theta_t^i=(Q_t^i, R_t^i)$ are private matrices and need to be reported to the system operator. 
However, we assume that it is \emph{a priori} known that $Q_t$ and $R_t$ satisfy
\begin{subequations}
\label{eqQRbound}
\begin{align}
&0\prec \underbar{Q} \preceq Q_t \preceq \overline{Q} \\
&0\prec \underbar{R} \preceq R_t \preceq \overline{R}
\end{align}
\end{subequations}
and are slowly time-varying in that
\begin{subequations}
\label{eqQRchange}
\begin{align}
&(1-\delta) Q_t \preceq Q_{t+1} \preceq (1+\delta) Q_t \\
&(1-\delta) R_t \preceq R_{t+1} \preceq (1+\delta) R_t 
\end{align}
\end{subequations}
with some small constant $\delta>0$.
We require that the reported sequence of matrices also satisfy (\ref{eqQRbound}) and (\ref{eqQRchange}).

Consider an MPC in which the social planner solves an open loop optimal control problem
\begin{equation}
\label{eqLQMPClti}
J_{t,T}(x_t;\theta_t)\triangleq \sum_{k=t}^{t+T-1} \left(x_k^\top Q_t x_k+u_k^\top R_t u_k\right).
\end{equation}
In (\ref{eqLQMPClti}), notice the weight matrices $Q_t$ and $R_t$ reported at time step $t$ are used over the entire horizon.
If $Q_t$ and $R_t$ vary sufficiently slowly, this is a reasonable MPC algorithm.
This MPC policy can be written as
\begin{align}
&\hat{u}_{t,T}(x_t)\triangleq \label{eqLQRHC}\\
& \argmin_{u_t} \left[x_t^\top Q_tx_t+u_t^\top R_t u_t + J_{t+1,T-1}(Ax_t+Bu_t;\theta_t)\right] \nonumber 
\end{align}
The cost-to-go incurred by the policy (\ref{eqLQRHC}) is denoted by
\begin{equation}
\label{eqMPCcost}
\hat{J}_{t,T}(x_t)=\sum_{k=t}^\infty x_k^\top Q_k x_k+\hat{u}_{k,T}^\top(x_k)R_k\hat{u}_{k,T}(x_k).
\end{equation}

\ifdefined\SHORTVERSION
Based on \cite{nevistic1997finite}, we can establish
\begin{equation}
\label{eqMPCexampleepsilon}
J_t(x_t)\leq \hat{J}_{t,T}(x_t)\leq (1+\epsilon_T) J_t(x_t)
\end{equation}
for every $(t,x_t)$, with some explicit expression of a constant $\epsilon_T$. See \cite{tanaka2016extended} for the details.
\fi
\ifdefined\LONGVERSION
Based on \cite{nevistic1997finite}, in Appendix we establish
\begin{equation}
\label{eqMPCexampleepsilon}
J_t(x_t)\leq \hat{J}_{t,T}(x_t)\leq (1+\epsilon_T) J_t(x_t)
\end{equation}
for every $(t,x_t)$, with some explicit expression of a constant $\epsilon_T$.
\fi
This inequality guarantees the $\epsilon$-efficiency of the considered MPC algorithm, and the $\epsilon$-incentive compatibility of the corresponding online mechanism $\mathcal{M}_{\text{VCG-MPC}}$.

Since $\epsilon_T$ tends to be small for large $T$, one can conclude that it is advantageous to use longer planning horizons to mitigate strategic misreporting. However, since MPC with long planning horizon is computational expensive, there is a trade-off between computational cost and incentive compatibility.

\section{Conclusion}
\label{sec:conclusions}
In this paper, motivated by  load frequency control in a power grid, we formulate the problem of designing online market mechanisms to incentivize strategic selfish entities that wish to provide frequency control services to the grid, to report their private information truthfully to the system operator. Using this private information, the system operator can use model predictive control to calculate control inputs that are socially optimal. The main challenge arises from the fact that every participant's strategy at any time affects the future state of other participants. Our solution is a VCG-like online mechanism that implements MPC in a way that guarantees that no strategic participant can gain by more than a specified bound by misreporting. 

Future work includes consideration of budget balance constraints in the formulation. It will also be of interest to include in the same framework the design of incentives for the participants to implement the control actions as well.

\ifdefined\LONGVERSION
\section*{Appendix: Proof of~(\ref{eqMPCexampleepsilon})}
We start with the following technical lemma.
\begin{lemma}
There exists a sequence $\{\alpha_T\}$ such that $\alpha_T>1$, $\lim_{T\rightarrow \infty} \alpha_T=1$, and
\begin{equation*}
\alpha_{T+1} J_{t,T}(x_t;\theta_t) \geq J_{t,T+1}(x_t;\theta_t) \;\; \forall t \;\forall x_t.
\end{equation*}
\end{lemma}
\begin{proof}
Notice that $J_{t,T}(x_t;\theta_t)=x_t^\top P_T x_t$ and $J_{t,T+1}(x_t;\theta_t)=x_t^\top P_{T+1} x_t$ where $P_T$ and $P_{T+1}$ are obtained by a Riccati recursion
\[
P_{k+1}=A^\top P_k A-A^\top P_k B (B^\top P_k B+R_t)^{-1}B^\top P_k A+Q_t
\]
with the initial condition $P_0=0$. Due to the convergence of the Riccati recursion, the claim clearly holds.
\end{proof}

Next, introduce a constant $0<\rho_T<1$ as the largest number that satisfies
\[
x_t^\top Q_t x_t \geq \rho_T J_{t,T}(x_t;\theta_t) \;\; \forall x_t \; \forall t.
\]
To compute $\rho_T$ explicitly, consider another Riccati recursion 
\[
\bar{P}_{k+1}=A^\top \bar{P}_k A-A^\top \bar{P}_k B (B^\top \bar{P}_k B+R_t)^{-1}B^\top \bar{P}_k A+Q_t
\]
with the initial condition $\bar{P}_0=0$. Since $R_t\leq \bar{R}$ and $Q_t\leq \bar{Q}$, due to the monotonicity of Riccati recursions, we have $\bar{P}_T \succeq P_T$. Set $\rho_T\triangleq \max\{\rho: \rho \bar{P}_T \preceq \underbar{Q}\}$. Then
\begin{align*}
x_t^\top Q_t x_t &\geq x_t^\top \underbar{Q} x_t \geq \rho_T x_t^\top \bar{P}_T x_t \\
&\geq \rho_T x_t^\top P_T x_t =\rho_T J_{t,T}(x_t;\theta_t).
\end{align*}


\begin{lemma}
Let $x_{0:\infty}$ and $u_{0:\infty}$ be the state and control trajectories resulting from the receding horizon control policy defined by (\ref{eqLQRHC}). 
\begin{itemize}
\item[(a)] $J_{t+1,T}(x_{t+1};\theta_{t+1}) \leq \gamma_T J_{t,T}(x_t;\theta_t)$ holds for every $t=0,1,\cdots$, where $\gamma_T\triangleq \frac{(1-\rho_T)\alpha_T}{(1-\delta)}$. 
\item[(b)] If $\gamma_T<1$, then the receding horizon control policy (\ref{eqLQRHC}) is stabilizing.
\item[(c)] If $\gamma_T<1$, then for every $t$, we have
\[
J_t(x_t)\!\leq\! \hat{J}_{t,T}(x_t)\!\leq\! \tfrac{\rho_T}{1-\gamma_T} J_{t,T}(x_t;\theta_t) \!\leq\! \tfrac{\rho_T(1-\delta)^{1-T}}{1-\gamma_T} J_t(x_t).
\]
\end{itemize}
\end{lemma}
\begin{proof}
(a). For every $t=0,1,\cdots$, we have
\begin{align*}
J_{t,T}(x_t;\theta_t)&=x_t^\top Q_tx_t+u_t^\top R_tu_t+J_{t+1,T-1}(x_{t+1};\theta_t) \\
&\geq \rho_T J_{t,T}(x_t;\theta_t)+J_{t+1,T-1}(x_{t+1};\theta_t) \\
&\geq \rho_T J_{t,T}(x_t;\theta_t)+(1-\delta)J_{t+1,T-1}(x_{t+1};\theta_{t+1}) \\
&\geq \rho_T J_{t,T}(x_t;\theta_t)+\tfrac{1-\delta}{\alpha_T}J_{t+1,T}(x_{t+1};\theta_{t+1}).
\end{align*}
Rearranging, we have $
\gamma_T J_{t,T}(x_t;\theta_t) \geq J_{t+1,T}(x_{t+1};\theta_{t+1})$.

(b). If $\gamma_T<1$, this implies that 
\[
\lim_{t\rightarrow \infty} J_{t,T}(x_t;\theta_t) \leq \lim_{t\rightarrow \infty} (\gamma_T)^t J_{0,T}(x_0;\theta_0)=0.
\]
Since $J_{t,T}(x_t;\theta_t) \geq x^\top \underbar{Q}x$ and $\underbar{Q}\succ 0$,
this proves  $\lim_{t\rightarrow \infty} x_t=0$.

(c). The first inequality is trivial, since (\ref{eqoptcost}) is the optimal cost-to-go, while (\ref{eqMPCcost}) is the cost-to-go attained by a suboptimal control polity (\ref{eqLQRHC}).
To see the second inequality, 
note that
\begin{align*}
&x_t^\top Q_tx_t+\hat{u}_{t,T}^\top(x_t)R_t\hat{u}_{t,T}(x_t)\\
&=J_{t,T}(x_t;\theta_t)-J_{t+1,T-1}(x_{t+1};\theta_t) \\
&\leq J_{t,T}(x_t;\theta_t)-(1-\delta)J_{t+1,T-1}(x_{t+1};\theta_{t+1}) \\
&\leq J_{t,T}(x_t;\theta_t)-\tfrac{1-\delta}{\alpha_T}J_{t+1,T}(x_{t+1};\theta_{t+1}) \\
&= J_{t,T}(x_t;\theta_t)-J_{t+1,T}(x_{t+1};\theta_{t+1}) \\
&\hspace{15ex}+\tfrac{\alpha_T+\delta-1}{\alpha_T}J_{t+1,T}(x_{t+1};\theta_{t+1}) 
\end{align*}
Similarly,
\begin{align*}
&x_{t+1}^\top Q_{t+1}x_{t+1}+\hat{u}_{t+1,T}^\top(x_{t+1}) R_{t+1}\hat{u}_{t+1,T}(x_{t+1}) \\
&\leq J_{t+1,T}(x_{t+1};\theta_{t+1})-J_{t+2,T}(x_{t+2};\theta_{t+2}) \\
&\hspace{15ex}+\tfrac{\alpha_T+\delta-1}{\alpha_T}J_{t+2,T}(x_{t+2};\theta_{t+2}).
\end{align*}
Thus,
\begin{align*}
\hat{J}_{t,T}&(x_t)=\sum_{k=t}^\infty x_k^\top Q_k x_k+\hat{u}_{k,T}^\top(x_k)R_k\hat{u}_{k,T}(x_k)  \\
& \leq J_{t,T}(x_t;\theta_t)+\left(\tfrac{\alpha_T+\delta-1}{\alpha_T}\right)\sum_{k=t}^\infty J_{k+1,T}(x_{k+1};\theta_{k+1})\\
& \leq J_{t,T}(x_t;\theta_t)+\left(\tfrac{\alpha_T+\delta-1}{\alpha_T}\right)\left(\sum_{l=1}^\infty \gamma_T^l\right)  J_{t,T}(x_t;\theta_t) \\
& \leq \left(1+\tfrac{\alpha_T+\delta-1}{\alpha_T}\tfrac{\gamma_T}{1-\gamma_T}\right)  J_{t,T}(x_t;\theta_t) \\
& =\tfrac{\rho_T}{1-\gamma_T}  J_{t,T}(x_t;\theta_t).
\end{align*}
To see the last inequality, note that
\begin{align*}
J_t(x_t) &\geq \inf_{u_{t:t+T-1}}\sum_{k=t}^{t+T-1}\left( x_k^\top Q_k x_k + u_k^\top R_k u_k \right) \\
&\geq  (1-\delta)^{T-1}\inf_{u_{t:t+T-1}}\sum_{k=t}^{t+T-1}\left( x_k^\top Q_t x_k + u_k^\top R_t u_k \right) \\
&=(1-\delta)^{T-1} J_{t,T}(x_t;\theta_t).
\end{align*}
The second inequality follows from the fact that
\[ (1-\delta)^{T-1}Q_t\preceq Q_k, \;\; (1-\delta)^{T-1}R_t\preceq R_k\]
for every $k=t,\cdots, t+T-1$. This is a consequence of the rate-of-change constraints (\ref{eqQRchange}).
\end{proof}

Finally, (\ref{eqMPCexampleepsilon}) is obtained by choosing $1+\epsilon_T=\tfrac{\rho_T(1-\delta)^{1-T}}{1-\gamma_T}$.

\fi


\bibliographystyle{IEEEtran}
\bibliography{Refs3}

\begin{thebibliography}{10}
\providecommand{\url}[1]{#1}
\csname url@samestyle\endcsname
\providecommand{\newblock}{\relax}
\providecommand{\bibinfo}[2]{#2}
\providecommand{\BIBentrySTDinterwordspacing}{\spaceskip=0pt\relax}
\providecommand{\BIBentryALTinterwordstretchfactor}{4}
\providecommand{\BIBentryALTinterwordspacing}{\spaceskip=\fontdimen2\font plus
\BIBentryALTinterwordstretchfactor\fontdimen3\font minus
  \fontdimen4\font\relax}
\providecommand{\BIBforeignlanguage}[2]{{%
\expandafter\ifx\csname l@#1\endcsname\relax
\typeout{** WARNING: IEEEtran.bst: No hyphenation pattern has been}%
\typeout{** loaded for the language `#1'. Using the pattern for}%
\typeout{** the default language instead.}%
\else
\language=\csname l@#1\endcsname
\fi
#2}}
\providecommand{\BIBdecl}{\relax}
\BIBdecl

\bibitem{ilic2007hierarchical}
M.~D. Ili{\'c}, ``From hierarchical to open access electric power systems,''
  \emph{Proceedings of the IEEE}, vol.~95, no.~5, pp. 1060--1084, 2007.

\bibitem{dorfler2014breaking}
F.~Dorfler, J.~Simpson-Porco, and F.~Bullo, ``Breaking the hierarchy:
  Distributed control \& economic optimality in microgrids,'' \emph{IEEE
  Transactions on Automatic Control (To appear)}, 2014.

\bibitem{bitar2012bringing}
E.~Y. Bitar, R.~Rajagopal, P.~P. Khargonekar, K.~Poolla, and P.~Varaiya,
  ``Bringing wind energy to market,'' \emph{IEEE Transactions on Power
  Systems}, vol.~27, no.~3, pp. 1225--1235, 2012.

\bibitem{doherty2010assessment}
R.~Doherty, A.~Mullane, G.~Nolan, D.~J. Burke, A.~Bryson, and M.~O'Malley, ``An
  assessment of the impact of wind generation on system frequency control,''
  \emph{IEEE Transactions on Power Systems}, vol.~25, no.~1, pp. 452--460,
  2010.

\bibitem{morren2006inertial}
J.~Morren, J.~Pierik, and S.~W. De~Haan, ``Inertial response of variable speed
  wind turbines,'' \emph{Electric power systems research}, vol.~76, no.~11, pp.
  980--987, 2006.

\bibitem{ulbig2013predictive}
A.~Ulbig, T.~Rinke, S.~Chatzivasileiadis, and G.~Andersson, ``Predictive
  control for real-time frequency regulation and rotational inertia provision
  in power systems,'' \emph{The 52nd IEEE Conference on Decision and Control
  (CDC)}, 2013.

\bibitem{venkat2008distributed}
A.~N. Venkat, I.~Hiskens, J.~B. Rawlings, S.~J. Wright \emph{et~al.},
  ``Distributed {MPC} strategies with application to power system automatic
  generation control,'' \emph{IEEE Transactions on Control Systems Technology},
  vol.~16, no.~6, pp. 1192--1206, 2008.

\bibitem{camacho2013model}
E.~F. Camacho and C.~B. Alba, \emph{Model predictive control}.\hskip 1em plus
  0.5em minus 0.4em\relax Springer Science \& Business Media, 2013.

\bibitem{xie2012fast}
L.~Xie, Y.~Gu, A.~Eskandari, and M.~Ehsani, ``Fast {MPC}-based coordination of
  wind power and battery energy storage systems,'' \emph{Journal of Energy
  Engineering}, vol. 138, no.~2, pp. 43--53, 2012.

\bibitem{arnold2011model}
M.~Arnold and G.~Andersson, ``Model predictive control of energy storage
  including uncertain forecasts,'' \emph{The 17th Power Systems Computation
  Conference (PSCC)}, 2011.

\bibitem{khalid2009model}
M.~Khalid and A.~V. Savkin, ``Model predictive control for wind power
  generation smoothing with controlled battery storage,'' \emph{The 48th IEEE
  Conference on Decision and Control (CDC)}, 2009.

\bibitem{berger1989real}
A.~W. Berger and F.~C. Schweppe, ``Real time pricing to assist in load
  frequency control,'' \emph{IEEE Transactions on Power Systems}, vol.~4,
  no.~3, pp. 920--926, 1989.

\bibitem{tanaka2012dynamic}
T.~Tanaka, A.~Z.~W. Cheng, and C.~Langbort, ``A dynamic pivot mechanism with
  application to real time pricing in power systems,'' \emph{The 2012 American
  Control Conference (ACC)}, 2012.

\bibitem{tanaka2013faithful}
T.~Tanaka, F.~Farokhi, and C.~Langbort, ``Faithful implementations of
  distributed algorithms and control laws,'' \emph{IEEE Transactions on Control
  of Network Systems (To appear)}, 2013.

\bibitem{mas1995microeconomic}
A.~Mas-Colell, M.~Whinston, and J.~Green, \emph{Microeconomic Theory}.\hskip
  1em plus 0.5em minus 0.4em\relax Oxford University Press, 1995.

\bibitem{vaheesan46market}
S.~Vaheesan, ``Market power in power markets: The filed-rate doctrine and
  competition in electricity,'' \emph{University of Michigan Journal of Law
  Reform}, vol.~46, p.~3.

\bibitem{hobbs2000evaluation}
B.~F. Hobbs, M.~H. Rothkopf, L.~C. Hyde, and R.~P. O'Neill, ``Evaluation of a
  truthful revelation auction in the context of energy markets with nonconcave
  benefits,'' \emph{Journal of Regulatory Economics}, vol.~18, no.~1, pp.
  5--32, 2000.

\bibitem{samadi2012advanced}
P.~Samadi, H.~Mohsenian-Rad, R.~Schober, and V.~W. Wong, ``Advanced demand side
  management for the future smart grid using mechanism design,'' \emph{IEEE
  Transactions on Smart Grid}, vol.~3, no.~3, pp. 1170--1180, 2012.

\bibitem{tang2011stochastic}
W.~Tang and R.~Jain, ``Stochastic resource auctions for renewable energy
  integration,'' \emph{The 49th Annual Allerton Conference on Communication,
  Control, and Computing}, 2011.

\bibitem{parkes2004mdp}
D.~C. Parkes and S.~Singh, ``An {MDP}-based approach to online mechanism
  design,'' \emph{Massachusetts Institute of Technology Press}, 2004.

\bibitem{bergemann2010dynamic}
D.~Bergemann and J.~Valimaki, ``The dynamic pivot mechanism,''
  \emph{Econometrica}, vol.~78, no.~2, pp. 771--789, 2010.

\bibitem{pavan2014dynamic}
A.~Pavan, I.~Segal, and J.~Toikka, ``Dynamic mechanism design: A {M}yersonian
  approach,'' \emph{Econometrica}, vol.~82, no.~2, pp. 601--653, 2014.

\bibitem{cavallo2012optimal}
R.~Cavallo, D.~C. Parkes, and S.~Singh, ``Optimal coordinated planning amongst
  self-interested agents with private state,'' \emph{arXiv preprint
  arXiv:1206.6820}, 2012.

\bibitem{tanaka2013faithful2}
T.~Tanaka, F.~Farokhi, and C.~Langbort, ``A faithful distributed implementation
  of dual decomposition and average consensus algorithms,'' \emph{The 52nd IEEE
  Conference on Decision and Control (CDC)}, 2013.

\bibitem{okajima2013dynamic}
Y.~Okajima, T.~Murao, K.~Hirata, and K.~Uchida, ``A dynamic mechanism for lqg
  power networks with random type parameters and pricing delay,'' \emph{The
  52nd IEEE Conference on Decision and Control (CDC)}, 2013.

\bibitem{murao2014dynamic}
T.~Murao, Y.~Okajima, K.~Hirata, and K.~Uchida, ``Dynamic balanced integration
  mechanism for {LQG} power networks with independent types,'' \emph{The 53rd
  IEEE Conference on Decision and Control (CDC)}, 2014.

\bibitem{andreasson2014distributed}
M.~Andreasson, D.~V. Dimarogonas, H.~Sandberg, and K.~H. Johansson,
  ``Distributed {PI}-control with applications to power systems frequency
  control,'' \emph{The 2014 American Control Conference (ACC)}, 2014.

\bibitem{shoham2008multiagent}
Y.~Shoham and K.~Leyton-Brown, \emph{Multiagent systems: Algorithmic,
  game-theoretic, and logical foundations}.\hskip 1em plus 0.5em minus
  0.4em\relax Cambridge University Press, 2008.

\bibitem{nevistic1997finite}
V.~Nevisti{\'c} and J.~A. Primbs, ``Finite receding horizon linear quadratic
  control: {A} unifying theory for stability and performance analysis,''
  \emph{California Institute of Technology}, 1997.

\end{thebibliography}

\end{document}